\documentclass[12pt,a4paper,reqno]{amsart}
\usepackage{amssymb}
\usepackage{amscd}
\usepackage{enumerate}
\usepackage{graphicx}
\usepackage{siunitx}
\usepackage{tikz-cd}
\usepackage{color}
\usetikzlibrary{arrows}
\numberwithin{equation}{section}

\usepackage{mathabx}

\usepackage{mathtools}
\usepackage[tableposition=top]{caption}
\usepackage{booktabs,dcolumn}

\DeclareFontFamily{OT1}{rsfs}{}
\DeclareFontShape{OT1}{rsfs}{n}{it}{<-> rsfs10}{}
\DeclareMathAlphabet{\mathscr}{OT1}{rsfs}{n}{it}

\theoremstyle{plain}

\newtheorem{theorem}{Theorem}[section]
\newtheorem{proposition}[theorem]{Proposition}
\newtheorem{lemma}[theorem]{Lemma}

\theoremstyle{definition}

\newtheorem{definition}[theorem]{Definition}
\newtheorem{remark}[theorem]{Remark}

\usepackage{algorithm, algorithmic}

\begin{document}

\title[Multiplicity of strong solutions]{Multiplicity of strong solutions for a class of elliptic problems
without the Ambrosetti-Rabinowitz condition in $\mathbb{R}^{N}$$^\star$}

\author{Li Yin}
\address{College of Information and Management Science, Henan Agricultural University,
Zhengzhou, Henan 450002, China}
\email{mathsr@163.com (L. Yin)}

\author{Jinghua Yao$^\dagger$}
\address{$^\dagger$ \textbf{Corresponding author}, Department of Mathematics, Indiana University, Bloomington, IN, 47408, USA.}
\email{yaoj@indiana.edu (J. Yao)}

\author{Qihu Zhang}
\address{Department of Mathematics and Information Science, Zhengzhou
University of Light Industry, Zhengzhou, Henan, 450002, China.}
\email{zhangqihu@yahoo.com (Q. Zhang)}

\author{Chunshan Zhao}
\address{Department of Mathematical Sciences, Georgia Southern University, Statesboro, GA 30460, USA.}
\email{czhao@GeorgiaSouthern.edu (C. Zhao)}
\date{\today}

\thanks{$\dagger$ Corresponding author}
\thanks{$^\star$This research is partly supported by
the key projects in Science and Technology Research of the Henan Education
Department (14A110011).}


\begin{abstract}
We investigate the existence and multiplicity of
solutions to the following $p(x)$-Laplacian problem in $\mathbb{R}^{N}$ via critical point theory
\begin{equation*}
\left\{
\begin{array}{l}
-\bigtriangleup _{p(x)}u+V(x)\left\vert u\right\vert ^{p(x)-2}u=f(x,u),\text{
in }
\mathbb{R}^{N}, \\
u\in W^{1,p(\cdot )}(\mathbb{R}^{N}).
\end{array}
\right.
\end{equation*}
We propose a new set of growth conditions which matches the variable exponent nature of the problem. Under this new set of assumptions, we manage to verify
the Cerami compactness condition. Therefore, we succeed in proving the existence of multiple
solutions to the above problem without the
well-known Ambrosetti--Rabinowitz type growth condition. Meanwhile, we could also characterize the pointwise asymptotic behaviors of these solutions. In our main argument, the idea of localization, decomposition of the domain, regularity of weak solutions and comparison principle are crucial ingredients among others.
\end{abstract}


\maketitle

\section{ Introduction}

In this paper, we study the existence and multiplicity of strong
solutions to a class of variable exponent problems in the whole space $\mathbb{R}^{N}$ as follows
\begin{equation*}
\tag{$P$}\left\{
\begin{array}{l}
-\bigtriangleup _{p(x)}u+V(x)\left\vert u\right\vert ^{p(x)-2}u=f(x,u),\text{
in }
\mathbb{R}^{N}, \\
u\in W^{1,p(\cdot )}(\mathbb{R}^{N}),
\end{array}
\right.
\end{equation*}
where $\Delta_{p(x)}$ is the $p(x)$-Laplace operator, also called $p(x)$-Laplacian, and its action on the unknown $u$ is defined by
$\Delta_{p(x)}u=\mbox{div}(|\nabla
u|^{p(x)-2}\nabla u)$, $V(x)$ is a nonlinear potential function, and $f=f(x, u)$ is the nonlinear term.

Before moving forward, let us first briefly recall some background related to the problem above. The study of differential equations and the related variational problems with variable
exponent has been a new and interesting topic. From an application point of view, this study has its backgrounds in such hot topics as
image processing, nonlinear electrorheological fluids and elastic mechanics. We refer the readers to \cite{e1}, \cite{e2}, \cite{e3}, \cite{wy}, \cite{e4} and the
references therein for more details. For a state of the art summary of the applications, see the monographs \cite{j5} and \cite{rr}. From a pure mathematical point of view, the study of differential equations involving the $p(x)$-Laplacian can be regarded as a natural further development of that involving $p$-Laplacian where $p\in (1, +\infty)$ is a constant. Compared with the classical Laplacian $\Delta$ ($p(\cdot)\equiv 2$) and the general $p$-Laplacian $\Delta_p$, the $p(x)$-Laplacian is both nonlinear and non-homogeneous. Therefore, the study of differential equations involving $p(x)$-Laplacian is usually more involved due to the nonlinear and non-homogeneous nature of  $p(x)$-Laplacian. 
Usually new methods and techniques are needed to study such elliptic
equations involving the non-standard growth, since the commonly known methods
and techniques to study elliptic equations involving standard growth may fail.
In our current investigation, we focus on a class of elliptic equations in the whole space $\mathbb R^N$ with the operator $-\Delta_{p(x)}\cdot +V(x)|\cdot|^{p(x)-2}\cdot$ in which a potential term $V(x)$ is also involved. Here we adopt the notion \textit{potential term} for $V(x)$ from the classical Schr$\ddot{o}$dinger operator $\mathcal H:=-\frac{\hbar^2}{2m}\Delta +V(x)$ where $\hbar$ is the Planck constant, $m$ mass of the particle, and $V(x)$ the potential (see \cite{sa}, \cite{wyl}). Later, we will see that this point of view also motivated our assumptions on $V(x)$ for our problem here. During the study of variational problems with variable exponent growth, the regularity properties of solutions to the related differential equations are also involved when one attempts to study finer properties of these solutions, see in particular \textit{Step 2} of the proof of Theorem 1.1 in the current study. There are a large number of 
related research works. We refer the readers to \cite{j8}, \cite{e11}, \cite{e12}, \cite{e14}, \cite{e16}, \cite{e17}, \cite{j1}, \cite{zj2},  \cite{yw}, \cite{e30}, \cite{e32} and the references therein for further details.

In this paper, our main goals are to study the existence and multiplicity of solutions
to the problem $(P)$ in the whole space setting $\mathbb R^N$ (unbounded in particular) without the
classical Ambrosetti-Rabinowitz condition, and to describe the pointwise asymptotic behavior of these solutions. For these purposes, we propose a new set of growth conditions under which we
are able to check the Cerami compactness condition. A main motivation of our current study is the observation that the famous Ambrosetti-Rabinowitz type condition sometimes can be very restrictive and
excludes many interesting nonlinearities. Actually, in the constant exponent case $p(\cdot
)\equiv p$, there were a series of papers dealing with problems without the
Ambrosetti-Rabinowitz type growth condition
(see \cite{zj4}, \cite{zj5}, \cite{jj5}, \cite{jj3},
\cite{jj2}, \cite{jj4}). However, due to the differences between the $p$-Laplacian and $p(x)$-Laplacian
mentioned above, often, it is very difficult to judge whether or
not results about $p$-Laplacian can be generalized to $p(x)$-Laplacian, and
even if the generalization can be done, it is still challenging to figure out the suitable or right forms in which
the results should be. To the best of our knowledge, there were rare results on variable exponent problem
without Ambrosetti-Rabinowitz type growth condition (see \cite{jj1}, \cite
{zj1}, \cite{31a}, \cite{zj3}). In an interesting earlier effort \cite{31a}, the author considered the
existence of solutions of variable exponent differential equations on a bounded domain, intending to weaken
the Ambrosetti-Rabinowitz growth condition. Unfortunately, the related assumption in \cite{31a} was actually stronger than the Ambrosetti-Rabinowitz condition. In a recent advancement \cite{jj1}, the authors considered
the similar problem as $(P)$ under a couple of interesting assumptions as follows: (1$^{0}$) There
exists a constant $\theta \geq 1$, such that $\theta \mathcal{F}(x,t)\geq
\mathcal{F}(x,st)$ for any$\ (x,t)\in \mathbb{R}^{N}\mathbb{\times R}$ and $
s\in \lbrack 0,1]$, where $\mathcal{F}(x,t)=f(x,t)t-p^{+}F(x,t)$; (2$^{0}$) $
f\in C(\mathbb{R}^{N}\mathbb{\times R},\mathbb{R})$ satisfies
$
\text{ }\underset{\left\vert t\right\vert \rightarrow \infty }{\lim }\frac{
F(x,t)}{\left\vert t\right\vert ^{p^{+}}}=\infty.$
Very recently in \cite{zj1}, the authors considered the problem $(P)$ in a bounded domain
under a condition similar to $(2^{0})$. Motivated by the above-mentioned interesting studies and \cite{zj3}, here we
study problem $(P)$ in the whole space $\mathbb R^N$ under growth conditions involving
variable growth rates which match much better the variable nature of the problem under
investigation.

\textbf{Notations.} In this paper, the notion ``$f_{1}(x)<<f_{2}(x)$" or equivalently ``$f_1<<f_2$" means that $\underset{x\in
\mathbb{R}
^{N}}{ess\inf }\,\big(f_{2}(x)-f_{1}(x)\big)>0$. We define the Sobolev conjugate exponent $p^{\ast }(x)$ of the variable exponent $p(x)$ as follows
\begin{equation*}
p^{\ast }(x)=\left\{
\begin{array}{l}
\frac{Np(x)}{N-p(x)}\text{, }p(x)<N, \\
\infty \qquad\text{, }p(x)\geq N.
\end{array}
\right.
\end{equation*}
Throughout the paper, we use letters $c,c_{i},C,C_{i}$, $i=1,2,...$ to
denote generic positive constants which may vary from line to line, and we
will specify them whenever it is necessary. For a function, say $p=p(x)$ or $f=f(x, t)$, we do not
distinguish the expressions $p$, $p(\cdot)$ and $p(x)$ or $f$, $f(\cdot, \cdot)$ and $f(x, t)$ when no ambiguity arises.
For sequences, we shall use ``$\rightharpoonup$" to represent weak convergence while ``$\rightarrow$" strong convergence in suitable function spaces.

To state our main results clearly and make the exposition concrete, we first collect our assumptions on the
potential function $V(x)$, the variable exponent $p(x)$ and the nonlinearity $f$ as follows:  

$(V)$ $V\in L_{loc}^{\infty }(
\mathbb{R}
^{N})$, $\underset{x\in\mathbb{R}^{N}}{ess\inf }\, V(x)=V_{0}>0$ and $V(x)\rightarrow +\infty $ as $\left\vert
x\right\vert \rightarrow \infty $.

$(p)$ $1<<p(\cdot )\in C^{1}(\mathbb{R}^{N})$ and $\left\vert \nabla p\right\vert\in L^{\infty}(\mathbb{R}^N)$.

$(H_{0})$ $f:\mathbb{R}^N\times \mathbb{R}\rightarrow \mathbb{R}$ satisfies the Caratheodory
condition and
\begin{equation*}
\left\vert f(x,t)\right\vert \leq C(\left\vert t\right\vert
^{p(x)-1}+\left\vert t\right\vert ^{\alpha (x)-1}),\forall (x,t)\in
\mathbb{R}^N\times \mathbb{R},
\end{equation*}
where $\alpha \in C(
\mathbb{R}^N)$ and $p(x)<<\alpha (x)<<p^{\ast }(x)$.

$(H_{1})$ There exist constants $M,C_{1},C_{2}>0$ and a function $a>>p$ on $\mathbb{R}^N$ such that
\begin{equation*}
C_{1}\left\vert t\right\vert ^{p(x)}[\ln (e+\left\vert t\right\vert
)]^{a(x)-1}\leq C_{2}\frac{tf(x,t)}{\ln (e+\left\vert t\right\vert )}\leq
tf(x,t)-p(x)F(x,t),\forall \left\vert t\right\vert \geq M,\forall x\in
\mathbb{R}^{N},
\end{equation*}
where and throughout the paper, $F(x, t):=\int_0^{t}f(x,s)\,ds$.

$(H_{2})$ $f(x,t)=o(\left\vert t\right\vert ^{p(x)-1})$ uniformly for $x\in
\mathbb R^N $, as $t\rightarrow 0$.

$(H_{3})$ $f(x,-t)=-f(x,t),$ $\forall x\in \mathbb R^N$, $\forall
t\in \mathbb{R}$.

Before stating our main results and proceeding further, we would like to briefly comment on some of the above assumptions. First, we point out that the assumption $(V)$ means that the potential
$V(x)$ is the standard potential well of infinite depth in classical quantum theory. From a mathematical point view, this natural potential well assumption yields exactly the needed compactness in our problem in the unbounded setting, see in particular Proposition \eqref{prop2.5}-$ii)$, Lemma \ref{lemma3.3}, and Lemma \ref{lemma3.4}. The assumption $(H_0)$ states that the nonlinearity $f$ here has a subcritical (in the sense of Sobolev embedding) growth rate in the whole space $\mathbb{R}^N\times \mathbb{R}$ while the assumptions $(H_1)$ and $(H_2)$ describe the far and near field behaviors of the nonlinearity $f$ with respect to its second argument respectively. It is worthing noting that the rates involved in the assumptions $(H_0)-(H_2)$ are all variable rates which match the variable nature of our problem under study.

Now we are in a position to state our main results.

\begin{theorem}
Suppose that the conditions (V), (p), (H$_{0}$)-(H$_{2}$%
) hold, then the problem (P) has two nontrivial strong solutions $u_{1}$ and
$u_{2}$ which satisfy $u_{1}>0>u_{2}$, $\underset{\left\vert x\right\vert
\rightarrow \infty }{\lim }u_{i}(x)=0$ and $\underset{\left\vert
x\right\vert \rightarrow \infty }{\lim }\left\vert \nabla
u_{i}(x)\right\vert =0$ for $i=1,2$.
\end{theorem}

\begin{theorem}
Suppose that the conditions (V), (p), (H$_{0}$), (H$%
_{1} $) and (H$_{3}$) hold, then there are infinitely many pairs of strong
solutions \{$\pm u_{i}$\} to the problem (P) which satisfy $\underset{%
\left\vert x\right\vert \rightarrow \infty }{\lim }u_{i}(x)=0$ and $\underset%
{\left\vert x\right\vert \rightarrow \infty }{\lim }\left\vert \nabla
u_{i}(x)\right\vert =0$ for $i=1,2,\cdots $.
\end{theorem}

For the rigorous definition of \textit{strong solutions} to problem $(P)$, see Definition
\ref{definition}. We postpone the rigorous introduction of definitions of solutions to the problem $(P)$ until Definition
\ref{definition}, as
it is only possible and more natural after we introduce the necessary functional-analytic framework.
Meanwhile. we have the following remark.

\begin{remark}
$(i)$ Let $f(x,t)=\left\vert t\right\vert ^{p(x)-2}t[\ln
(1+\left\vert t\right\vert )]^{a(x)}$, then we can see $f$ satisfies the
condition $(H_{0})-(H_{3})$, but it does not satisfy the
Ambrosetti-Rabinowitz condition. $(ii)$ We do not need any monotone
assumption on the nonlinearity $f$ in our main results.
\end{remark}

The rest of the paper is organized as follows. In Section 2, we recall for our study the minimal functional-analytic preparation
related to the variable exponent Lebesgue and Sobolev spaces. Some useful lemmas are also included in this section. In Section 3, we give the proofs of the above main results.

\section{Preliminaries}

In order to discuss problem $(P)$ rigorously, we recall here some necessary results on the variable exponent Lebesgue space
$L^{p(\cdot )}(\mathbb{R}^{N})$ and Sobolev space $W^{1,p(\cdot )}(
\mathbb{R}^{N})$. For more systematic information, please see the monographs \cite{j5}, \cite{rr}, papers \cite{e9}, \cite{e12}, \cite{e15}, \cite{e20},  \cite{e26}, \cite{yao} and the references therein. In this section, we shall validate the assumptions $(V)$ and $(p)$.

We denote $C_{+}(\mathbb{R}^{N}) =\{h\in C(\mathbb{R}^{N})\, \big|\, h(x)>>1\}$.
For $h\in C(\mathbb{R}
^{N})$, we define
$$h^{+}=\underset{
\mathbb{R}
^{N}}{\sup }\,h(x),\quad h^{-}=\underset{\mathbb{R}
^{N}}{\inf }\, h(x).$$
Now we introduce the variable exponent Lebesgue space $L^{p(\cdot )}(\mathbb{R}^{N})$ as follows:
$$L^{p(\cdot )}(
\mathbb{R}
^{N}) =\left\{ u\mid u\text{ is a measurabled real-value function, }\int_{\mathbb{R}
^{N}}\left\vert u(x)\right\vert ^{p(x)}dx<\infty \right\}.$$

On $%
L^{p(\cdot )}(
\mathbb{R}^{N})$, we introduce the Luxemberg norm  by

\begin{center}
$\left\vert u\right\vert _{p(\cdot )}=$ $\inf \left\{ \lambda >0\left\vert
\int_{\mathbb{R}
^{N}}\left\vert \frac{u(x)}{\lambda }\right\vert ^{p(x)}dx\leq 1\right.
\right\} $.
\end{center}

Then ($L^{p(\cdot )}(\mathbb{R}
^{N})$, $\left\vert \cdot \right\vert _{p(\cdot )}$) becomes a Banach space
and it is the so-called variable exponent Lebesgue space.\newline

For the variable exponent Lebesgue spaces, we have the following version of H$\ddot{o}$lder's inequality and embedding.
\begin{proposition}
(see \cite{j5}, \cite{e9}, \cite{e12}). i) The
space $(L^{p(\cdot )}(
\mathbb{R}
^{N}),\left\vert \cdot \right\vert _{p(\cdot )})$ is a separable, uniform
convex Banach space, and its conjugate space is $L^{q(\cdot )}(\mathbb{R}
^{N}),$ where $\frac{1}{q(x)}+\frac{1}{p(x)}\equiv 1$. For any $u\in
L^{p(\cdot )}(
\mathbb{R}
^{N})$ and $v\in L^{q(\cdot )}(\mathbb{R}
^{N})$, we have
\begin{equation*}
\left\vert \int_{
\mathbb{R}
^{N}}uvdx\right\vert \leq (\frac{1}{p^{-}}+\frac{1}{q^{-}})\left\vert
u\right\vert _{p(\cdot )}\left\vert v\right\vert _{q(\cdot )}.
\end{equation*}

ii) If $p_{1},$ $p_{2}\in C_{+}(
\mathbb{R}
^{N})$, $\Omega \subset
\mathbb{R}
^{N}$ is a bounded domain, $p_{1}(x)\leq p_{2}(x)$ for any $x\in \Omega ,$
then $L^{p_{2}(\cdot )}(\Omega )\subset L^{p_{1}(\cdot )}(\Omega ),$ and the
embedding is continuous.
\end{proposition}

In the study of nonlinear elliptical variational problems, the property of the Nemytsky operator
plays an important role. In the variable exponent Lebesgue spaces framework, we have the following
property concerning the Nemytsky operator.

\begin{proposition}
(see \cite{e12}). If $f:$ $\mathbb{R}
^{N}\times
\mathbb{R}
\rightarrow
\mathbb{R}
$ is a Caratheodory function and satisfies
\begin{equation*}
\left\vert f(x,s)\right\vert \leq h(x)+b\left\vert s\right\vert
^{p_{1}(x)/p_{2}(x)}\ \text{\ for any }x\in \mathbb R^N ,s\in
\mathbb{R},
\end{equation*}
where $p_{1}$, $p_{2}\in C_{+}(\mathbb{R}
^{N})$ , $h\in L^{p_{2}(\cdot )}(\mathbb{R}
^{N})$, $h(x)\geq 0$, $b\geq 0$, then the Nemytsky operator from $
L^{p_{1}(\cdot )}(\mathbb R^N )$ to $L^{p_{2}(\cdot )}(\mathbb R^N)$ defined by $
(N_{f}u)(x)=f(x,u(x))$ is a continuous and bounded operator.
\end{proposition}

Compared with the constant exponent spaces, one of the distinct features of variable exponent spaces is that there is no strict equality relation between the module and the norm of a function in the variable exponent spaces. Related to this feature, we have the following two propositions.

\begin{proposition}
(see \cite{e12}). If we denote
\begin{equation*}
\rho (u)=\int_{
\mathbb{R}
^{N}}\left\vert u\right\vert ^{p(x)}dx\text{, }\forall u\in L^{p(\cdot )}(\mathbb{R}^{N}),
\end{equation*}
then there exists a $\xi \in
\mathbb{R}^{N}$ such that $\left\vert u\right\vert _{p(\cdot )}^{p(\xi )}=\int_{\mathbb{R}
^{N}}\left\vert u\right\vert ^{p(x)}dx$ and

i) $\left\vert u\right\vert _{p(\cdot )}<1(=1;>1)\Longleftrightarrow \rho
(u)<1(=1;>1);$

ii) $\left\vert u\right\vert _{p(\cdot )}>1\Longrightarrow \left\vert
u\right\vert _{p(\cdot )}^{p^{-}}\leq \rho (u)\leq \left\vert u\right\vert
_{p(\cdot )}^{p^{+}};$ $\left\vert u\right\vert _{p(\cdot
)}<1\Longrightarrow \left\vert u\right\vert _{p(\cdot )}^{p^{-}}\geq \rho
(u)\geq \left\vert u\right\vert _{p(\cdot )}^{p^{+}};$

iii) $\left\vert u\right\vert _{p(\cdot )}\rightarrow 0\Longleftrightarrow
\rho (u)\rightarrow 0;$ $\left\vert u\right\vert _{p(\cdot )}\rightarrow
\infty \Longleftrightarrow \rho (u)\rightarrow \infty$.
\end{proposition}

\begin{proposition}
(see \cite{e12}). If $u$, $u_{n}\in L^{p(\cdot )}(\mathbb{R}^{N})$, $n=1,2,\cdots ,$ then the following statements are equivalent to
each other.

1) $\underset{k\rightarrow \infty }{\lim }$ $\left\vert u_{k}-u\right\vert
_{p(\cdot )}=0;$

2) $\underset{k\rightarrow \infty }{\lim }$ $\rho \left( u_{k}-u\right) =0;$

3) $u_{k}$ $\rightarrow $ $u$ in measure in $\mathbb{R}^{N}$ and $\underset{k\rightarrow \infty }{\lim }$ $\rho \left( u_{k}\right)
=\rho (u).$
\end{proposition}

Now we introduce the variable exponent Sobolev space $W^{1,p(\cdot )}(
\mathbb{R}^{N})$. The space $W^{1,p(\cdot )}(
\mathbb{R}^{N})$ is defined through
\begin{equation*}
W^{1,p(\cdot )}(\mathbb{R}^{N})=\left\{ u\in L^{p(\cdot )}\left(
\mathbb{R}^{N}\right) \left\vert \nabla u\in (L^{p(\cdot )}\left(
\mathbb{R}^{N}\right) )^{N}\right. \right\} ,
\end{equation*}
equipped with the following norm
\begin{equation*}
\left\Vert u\right\Vert _{p(\cdot )}=\left\vert u\right\vert _{p(\cdot
)}+\left\vert \nabla u\right\vert _{p(\cdot )},\forall u\in W^{1,p(\cdot
)}\left(\mathbb{R}^{N}\right) .
\end{equation*}

Adapted to our specific problem here, we introduce the working space $X$ by
\begin{equation*}
X=\left\{ u\in W^{1,p(\cdot )}(
\mathbb{R}
^{N})\left\vert \int_{
\mathbb{R}^{N}}\left\vert \nabla u(x)\right\vert ^{p(x)}+V(x)\left\vert
u(x)\right\vert ^{p(x)}dx<\infty \right. \right\} ,
\end{equation*}
and it can be equipped with the norm
\begin{equation*}
\left\Vert u\right\Vert =\inf \left\{ \lambda >0\left\vert \int_{\mathbb{R}^{N}}\left\vert \frac{\nabla u(x)}{\lambda }\right\vert
^{p(x)}+V(x)\left\vert \frac{u(x)}{\lambda }\right\vert ^{p(x)}dx\leq
1\right. \right\} .
\end{equation*}
Obviously, $X$ is a closed linear subspace of $W^{1,p(\cdot )}(\mathbb{R}^{N})$.  

Concerning the space $X$, we also have the following proposition.

\begin{proposition}\label{prop2.5}
(see \cite{jj1} Lemma 2.6). i) $X$ is a separable
reflexive Banach spaces;

ii) If $q\in C_{+}\left(\mathbb{R}^{N}\right) $ and $p(x)\leq q(x)<<p^{\ast }(x)$ for any $x\in
\mathbb{R}^{N}$, then the embedding from $X$ to $L^{q(\cdot )}\left(\mathbb{R}^{N}\right) $ is compact and continuous.
\end{proposition}

Next we give some results related to the $p(x)$-Laplacian $\bigtriangleup _{p(x)}$. Consider
the following functional
\begin{equation*}
J(u)=\int_{\mathbb{R}^{N}}\frac{1}{p(x)}(\left\vert \nabla u\right\vert ^{p(x)}+V(x)\left\vert
u\right\vert ^{p(x)})dx,\text{ }u\in X.
\end{equation*}
Obviously (see \cite{e39}), $J\in C^{1}(X,\mathbb{R}),$ and the operator $-\Delta_{p(x)}\cdot +V(x)|\cdot|^{p(x)-2}\cdot$ is the derivative operator of $J$ in the weak sense. We
denote $L:=J^{^{\prime }}:X\rightarrow X^{\ast }$, then
\begin{equation*}
(L(u),v)=\int_{\mathbb{R}^{N}}(\left\vert \nabla u\right\vert ^{p(x)-2}\nabla u\nabla
v+V(x)\left\vert u\right\vert ^{p(x)-2}uv)dx\text{, }\forall u,v\in X.
\end{equation*}

On the properties of the derivative operator $L$, we have the following proposition which will be useful in Sections 3.

\begin{proposition}
(see \cite{e12}, \cite{e17}). i) $L:X\rightarrow
X^{\ast }$ is a continuous, bounded and strictly monotone operator;

ii) $L$ is a mapping of type $(S_{+})$, i.e., if $u_{n}\rightharpoonup u$ in
$X$ and $\underset{n\rightarrow +\infty }{\overline{\lim }}$ $
(L(u_{n})-L(u),u_{n}-u)\leq 0$, then $u_{n}\rightarrow u$ in $X$;

iii) $L:X\rightarrow X^{*}$ is a homeomorphism.
\end{proposition}

Now, we denote $B(x_{0},\varepsilon ,\delta ,\theta ):=\{x\in \mathbb{R} ^{N}\mid
\delta \leq \left\vert x-x_{0}\right\vert \leq \varepsilon ,\frac{x-x_{0}}{
\left\vert x-x_{0}\right\vert }\cdot \frac{\nabla p(x_{0})}{\left\vert
\nabla p(x_{0})\right\vert }\geq \cos \theta \}$, where $\theta \in (0,\frac{
\pi }{2})$. Then, we have the following geometrical proposition on the variable exponent $p(x)$. 

\begin{lemma}
(see \cite{zj3}) If $p\in C^{1}(\mathbb{R}^{N})$, $
x_{0}\in \mathbb{R}^{N}$ satisfy $\nabla p(x_{0})\neq 0$, then there exist a
positive $\varepsilon $ small enough such that
\begin{equation}
(x-x_{0})\cdot \nabla p(x)>0,\forall x\in B(x_{0},\varepsilon ,\delta
,\theta )\text{,}  \label{a3}
\end{equation}
and
\begin{equation}
\max \{p(x)\mid x\in \overline{B(x_{0},\varepsilon )}\}=\max \{p(x)\mid x\in
B(x_{0},\varepsilon ,\varepsilon ,\theta )\}.  \label{a4}
\end{equation}
\end{lemma}

\begin{proof} The proof can be found in \cite{zj3}. We reproduce it here for completeness and for the readers'
convenience.

Since $p\in C^{1}(\mathbb{R}^{N})$, for any $x\in B(x_{0},\varepsilon
,\delta ,\theta )$, when $\varepsilon $ is small enough, it is easy to see
that
\begin{eqnarray*}
\nabla p(x)\cdot (x-x_{0}) &=&(\nabla p(x_{0})+o(1))\cdot (x-x_{0}) \\
&=&\nabla p(x_{0})\cdot (x-x_{0})+o(\left\vert x-x_{0}\right\vert ) \\
&\geq &\left\vert \nabla p(x_{0})\right\vert \left\vert x-x_{0}\right\vert
\cos \theta +o(\left\vert x-x_{0}\right\vert )>0,
\end{eqnarray*}
where $o(1)\in
\mathbb{R}^{N}$ is a function and $o(1)\rightarrow 0$ uniformly as $\left\vert
x-x_{0}\right\vert \rightarrow 0$.

When $\varepsilon $ is small enough, (\ref{a3}) is valid. Since $p\in C^{1}(
\mathbb{R}^{N})$, there exist a small enough positive $\varepsilon $ such
that
\begin{equation*}
p(x)-p(x_{0})=\nabla p(y)\cdot (x-x_{0})=(\nabla p(x_{0})+o(1))\cdot
(x-x_{0}),\text{ }
\end{equation*}
where $y=x_{0}+\tau (x-x_{0})$ and $\tau \in (0,1)$, $o(1)\in\mathbb{R}^{N}$ is a function and $o(1)\rightarrow 0$ uniformly as $\left\vert
x-x_{0}\right\vert \rightarrow 0$.

Suppose $x\in \overline{B(x_{0},\varepsilon )}\backslash B(x_{0},\varepsilon
,\delta ,\theta )$. Denote $x^{\ast }=x_{0}+\varepsilon \nabla
p(x_{0})/\left\vert \nabla p(x_{0})\right\vert $.

Suppose $\frac{x-x_{0}}{\left\vert x-x_{0}\right\vert }\cdot \frac{\nabla
p(x_{0})}{\left\vert \nabla p(x_{0})\right\vert }<\cos \theta $. When $
\varepsilon $ is small enough, we have
\begin{eqnarray*}
p(x)-p(x_{0}) &=&(\nabla p(x_{0})+o(1))\cdot (x-x_{0}) \\
&<&\left\vert \nabla p(x_{0})\right\vert \left\vert x-x_{0}\right\vert \cos
\theta +o(\varepsilon ) \\
&\leq &(\nabla p(x_{0})+o(1))\cdot \varepsilon \nabla p(x_{0})/\left\vert
\nabla p(x_{0})\right\vert \\
&=&p(x^{\ast })-p(x_{0}),
\end{eqnarray*}
where $o(1)\in\mathbb{R}^N$ is a function and $o(1)\rightarrow 0$ as $\varepsilon \rightarrow 0$.

Suppose $\left\vert x-x_{0}\right\vert <\delta $. When $\varepsilon $ is
small enough, we have
\begin{eqnarray*}
p(x)-p(x_{0}) &=&(\nabla p(x_{0})+o(1))\cdot (x-x_{0}) \\
&\leq &\left\vert \nabla p(x_{0})\right\vert \left\vert x-x_{0}\right\vert
+o(\varepsilon ) \\
&<&(\nabla p(x_{0})+o(1))\cdot \varepsilon \nabla p(x_{0})/\left\vert \nabla
p(x_{0})\right\vert \\
&=&p(x^{\ast })-p(x_{0}),
\end{eqnarray*}
where $o(1)\in\mathbb{R}^{N}$ is a function and $o(1)\rightarrow 0$ as $\varepsilon \rightarrow 0$.
Thus, we have
\begin{equation}
\max \{p(x)\mid x\in \overline{B(x_{0},\varepsilon )}\}=\max \{p(x)\mid x\in
B(x_{0},\varepsilon ,\delta ,\theta )\}.  \label{d1}
\end{equation}

It follows from (\ref{a3}) and (\ref{d1}) that (\ref{a4}) is valid. The proof of
Lemma 2.8 is completed.
\end{proof}

Next, we give a lemma which will also be used in Section 3.

\begin{lemma}
Suppose that the function $F(x,u)$ satisfies the following far-field growth condition with respect to its second argument $u$:
\begin{equation*}
C_{1}\left\vert u\right\vert ^{p(x)}[\ln (e+\left\vert u\right\vert
)]^{a(x)}\leq F(x,u),\forall \left\vert u\right\vert \geq M,\forall x\in
\mathbb{R}^{N},
\end{equation*}
where $a(x)>>p(x)$, and $x_{0}\in \mathbb{R}^{N}$ with $\nabla p(x_{0})\neq
0 $. Let
\begin{equation*}
h(x)=\left\{
\begin{array}{cc}
0, & \left\vert x-x_{0}\right\vert >\varepsilon \\
\varepsilon -\left\vert x-x_{0}\right\vert , & \left\vert x-x_{0}\right\vert
\leq \varepsilon
\end{array}
\right. ,
\end{equation*}
where $\varepsilon $ is defined in Lemma 2.7, then there holds
\begin{equation*}
\varphi (th)=\int_{\mathbb R^N }\frac{1}{p(x)}(\left\vert \nabla th\right\vert
^{p(x)}+V(x)\left\vert th\right\vert ^{p(x)})dx-\int_{\mathbb R^N
}F(x,th)dx\rightarrow -\infty \text{ as }t\rightarrow +\infty .
\end{equation*}
\end{lemma}

\begin{proof} We only need to prove that
\begin{align*}
\Psi (th)&:=\int_{\mathbb R^N }\frac{1}{p(x)}(\left\vert \nabla th\right\vert
^{p(x)}+V(x)\left\vert th\right\vert ^{p(x)})\,dx-\int_{\mathbb R^N }C_{1}\left\vert th\right\vert ^{p(x)}[\ln
(e+\left\vert th\right\vert )]^{a(x)}dx\\&
\rightarrow -\infty \text{ as }
t\rightarrow +\infty .
\end{align*}
Obviously, for $t$ sufficiently large, we have
\begin{equation*}
\int_{\mathbb R^N }\frac{1}{p(x)}\left\vert \nabla th\right\vert ^{p(x)}dx\leq
C_{2}\int_{B(x_{0},\varepsilon ,\delta ,\theta )}\left\vert \nabla
th\right\vert ^{p(x)}dx,
\end{equation*}
\begin{equation*}
\int_{\mathbb R^N }V(x)\left\vert th\right\vert ^{p(x)}dx\leq
C_{2}\int_{B(x_{0},\varepsilon ,\delta ,\theta )}V(x)\left\vert
th\right\vert ^{p(x)}dx,
\end{equation*}
\begin{equation*}
\int_{\mathbb R^N }C_{1}\left\vert th\right\vert ^{p(x)}[\ln (e+\left\vert
th\right\vert )]^{a(x)}dx\geq \int_{B(x_{0},\varepsilon ,\delta ,\theta
)}C_{1}\left\vert th\right\vert ^{p(x)}[\ln (e+\left\vert th\right\vert
)]^{a(x)}dx.
\end{equation*}
Next we shall use spherical coordinates. Denote $r=\left\vert
x-x_{0}\right\vert $. Since $p\in C^{1}(\mathbb{R}^{N})$, it follows from (\ref{a3}) that there exist positive constants $c_{1}$ and $c_{2}$ such that
\begin{equation*}
p(\varepsilon ,\omega )-c_{2}(\varepsilon -r)\leq p(r,\omega )\leq
p(\varepsilon ,\omega )-c_{1}(\varepsilon -r),\forall (r,\omega )\in
B(x_{0},\varepsilon ,\delta ,\theta ).
\end{equation*}
Therefore, we have
\begin{eqnarray} \label{a5}
\int_{B(x_{0},\varepsilon ,\delta ,\theta )}\left\vert \nabla th\right\vert
^{p(x)}dx &=&\int_{B(x_{0},\varepsilon ,\delta ,\theta )}t^{p(r,\omega
)}r^{N-1}drd\omega  \notag \\
&\leq &\int_{B(x_{0},\varepsilon ,\delta ,\theta )}t^{p(\varepsilon ,\omega
)-c_{1}(\varepsilon -r)}r^{N-1}drd\omega  \notag \\
&\leq &\varepsilon ^{N-1}\int_{B(x_{0},\varepsilon ,\delta ,\theta
)}t^{p(\varepsilon ,\omega )-c_{1}(\varepsilon -r)}drd\omega  \notag \\
&\leq &\varepsilon ^{N-1}\int_{B(x_{0},1,1,\theta )}\frac{t^{p(\varepsilon
,\omega )}}{c_{1}\ln t}d\omega.
\end{eqnarray}
Similarly, we have
\begin{eqnarray} \label{abc2}
\int_{B(x_{0},\varepsilon ,\delta ,\theta )}V(x)\left\vert th\right\vert
^{p(x)}dx &=&\int_{B(x_{0},\varepsilon ,\delta ,\theta )}V(r,\omega
)\left\vert t(\varepsilon -r)\right\vert ^{p(r,\omega )}r^{N-1}drd\omega
\notag \\
&\leq &c_{3}\int_{B(x_{0},\varepsilon ,\delta ,\theta )}\left\vert
t\right\vert ^{p(r,\omega )}drd\omega  \notag \\
&\leq &c_{3}\int_{B(x_{0},\varepsilon ,\delta ,\theta )}\left\vert
t\right\vert ^{p(\varepsilon ,\omega )-c_{1}(\varepsilon -r)}drd\omega
\notag \\
&=&c_{3}\int_{B(x_{0},1,1,\theta )}d\omega \int_{\delta }^{\varepsilon
}\left\vert t\right\vert ^{p(\varepsilon ,\omega )-c_{1}(\varepsilon -r)}dr
\notag \\
&\leq &c_{4}\int_{B(x_{0},1,1,\theta )}\frac{\left\vert t\right\vert
^{p(\varepsilon ,\omega )}}{\ln t}d\omega. 
\end{eqnarray}
Since $p\in C^{1}(\mathbb{R}^{N})$, and $a(\cdot )>>p(\cdot )$, when $
\varepsilon $ is small enough, there exists a $\epsilon _{1}>0$ such that
\begin{equation*}
a(x)\geq \max \{p(x)+\epsilon _{1}\mid x\in B(x_{0},\varepsilon ,\delta
,\theta )\},\forall x\in B(x_{0},\varepsilon ,\delta ,\theta ).
\end{equation*}
Hence, when $t$ is large enough, we have
\begin{eqnarray*}
&&\int_{B(x_{0},\varepsilon ,\delta ,\theta )}C_{1}\left\vert th\right\vert
^{p(x)}[\ln (e+\left\vert th\right\vert )]^{a(x)}dx \\
&=&\int_{B(x_{0},\varepsilon ,\delta ,\theta )}C_{1}\left\vert t(\varepsilon
-r)\right\vert ^{p(r,\omega )}r^{N-1}[\ln (e+\left\vert t(\varepsilon
-r)\right\vert )]^{a(r,\omega )}drd\omega \\
&\geq &C_{1}\delta ^{N-1}\int_{B(x_{0},\varepsilon ,\delta ,\theta
)}\left\vert t\right\vert ^{p(\varepsilon ,\omega )-c_{2}(\varepsilon
-r)}\left\vert \varepsilon -r\right\vert ^{p(\varepsilon ,\omega
)-c_{1}(\varepsilon -r)}[\ln (e+\left\vert t(\varepsilon -r)\right\vert
)]^{a(r,\omega )}drd\omega \\
&\geq &C_{1}\delta ^{N-1}\int_{B(x_{0},1,1,\theta )}d\omega \int_{\delta
}^{\varepsilon -\frac{1}{\ln t}}\left\vert t\right\vert ^{p(\varepsilon
,\omega )-c_{2}(\varepsilon -r)}\left\vert \varepsilon -r\right\vert
^{p(\varepsilon ,\omega )-c_{1}(\varepsilon -r)}[\ln (e+\left\vert
t(\varepsilon -r)\right\vert )]^{a(r,\omega )}dr \\
&\geq &C_{3}\delta ^{N-1}\int_{B(x_{0},1,1,\theta )}(\frac{1}{\ln t}
)^{p(\varepsilon ,\omega )}[\ln (e+\frac{t}{\ln t})]^{p(\varepsilon ,\omega
)+\epsilon _{1}}\int_{\delta }^{\varepsilon -\frac{1}{\ln t}}\left\vert
t\right\vert ^{p(\varepsilon ,\omega )-c_{2}(\varepsilon -r)}drd\omega \\
&\geq &C_{4}\delta ^{N-1}\int_{B(x_{0},1,1,\theta )}(\ln t)^{\epsilon _{1}}
\frac{\left\vert t\right\vert ^{p(\varepsilon ,\omega )-\frac{c_{2}}{\ln t}}
}{c_{2}\ln t}d\omega \\
&\geq &(\ln t)^{\epsilon _{1}}C_{5}\delta ^{N-1}\int_{B(x_{0},1,1,\theta )}
\frac{\left\vert t\right\vert ^{p(\varepsilon ,\omega )}}{c_{2}\ln t}d\omega.
\end{eqnarray*}
Thus, we have
\begin{equation}\label{a6}
\int_{B(x_{0},\varepsilon ,\delta ,\theta )}C_{1}\left\vert th\right\vert
^{p(x)}[\ln (e+\left\vert th\right\vert )]^{a(x)}dx\geq (\ln t)^{\epsilon
_{1}}C_{5}\int_{B(x_{0},1,1,\theta )}\frac{\left\vert t\right\vert
^{p(\varepsilon ,\omega )}}{\ln t}d\omega \text{ as }t\rightarrow +\infty .
\end{equation}
It follows from (\ref{a5}), (\ref{abc2}) and (\ref{a6}) that $\Psi
(th)\rightarrow -\infty $. The proof of Lemma 2.8 is completed.
\end{proof}

\section{Proofs of main results}

In this section, we will prove the main results. We first make clear the definitions of weak and strong solutions
to the problem $(P)$ respectively.

\begin{definition}\label{definition}
$(i)$ We call $u\in X$ is a weak solution of $(P)$ if
\begin{equation*}
\int_{\mathbb{R}^{N}}\left\vert \nabla u\right\vert ^{p(x)-2}\nabla u\cdot
\nabla v+V(x)\left\vert u\right\vert ^{p(x)-2}uvdx=\int_{\mathbb{R}
^{N}}f(x,u)vdx,\forall v\in X.
\end{equation*}

$(ii)$ We call $u\ $is a strong solution of $(P)$ if $u\in C^{1,\alpha }(\mathbb{
R}^{N})$ and it is a weak solution of $(P)$.
\end{definition}

The corresponding functional of $(P)$ is
\begin{equation*}
\varphi \left( u\right) =\int_{\mathbb{R}^{N}}\frac{1}{p(x)}(\left\vert
\nabla u\right\vert ^{p(x)}+V(x)\left\vert u\right\vert ^{p(x)})dx-\int_{
\mathbb{R}^{N}}F(x,u)dx,\forall u\in X.
\end{equation*}

\begin{definition}
We say $\varphi $ satisfies the Cerami condition in $X$,
if any sequence $\left\{ u_{n}\right\} \subset X$ such that $\left\{ \varphi
(u_{n})\right\} $ is bounded and $\left\Vert \varphi ^{\prime
}(u_{n})\right\Vert (1+\left\Vert u_{n}\right\Vert )\rightarrow 0$ as $
n\rightarrow +\infty $ has a strong convergent subsequence in $X$.
\end{definition}

\begin{lemma}\label{lemma3.3}
Suppose that the conditions (V), (p), (H$_{0}$)-(H$_{1}$)
hold, then $\varphi $ satisfies the Cerami condition in $X$.
\end{lemma}
\begin{proof}
Let $\{u_{n}\}\subset X$ be a Cerami sequence, i.e., $\varphi
(u_{n})\rightarrow c$ and $\left\Vert \varphi ^{\prime }(u_{n})\right\Vert
(1+\left\Vert u_{n}\right\Vert )\rightarrow 0$. Suppose $\{u_{n}\}$ is
bounded, then $\{u_{n}\}$ has a weakly convergent subsequence in $X$.
Without loss of generality, we assume that $u_{n}\rightharpoonup u$ in $X$, according to Proposition \ref{prop2.5}-$ii$), then $
f(x,u_{n})\rightarrow f(x,u)$ in $X^{\ast }$. Since $\varphi ^{\prime
}(u_{n})=L(u_{n})-f(x,u_{n})\rightarrow 0$ in $X^{\ast }$, we have $
L(u_{n})\rightarrow f(x,u)$ in $X^{\ast }$. Since $L$ is a homeomorphism, we
have $u_{n}\rightarrow u$ in $X$, and therefore $\varphi $ satisfies Cerami condition.
Hence, we only need to prove the boundedness of the Cerami sequence $\{u_{n}\}$.

We suppose, on the contrary, that there exist $c\in \mathbb{R}$ and $
\{u_{n}\}\subset X$ satisfying:
\begin{equation*}
\varphi (u_{n})\rightarrow c\text{, }\left\Vert \varphi ^{\prime
}(u_{n})\right\Vert (1+\left\Vert u_{n}\right\Vert )\rightarrow 0\text{, }
\left\Vert u_{n}\right\Vert \rightarrow +\infty .
\end{equation*}

Obviously, there holds
\begin{equation*}
\left\vert \frac{1}{p(x)}u_{n}\right\vert _{p(\cdot )}\leq \frac{1}{p^{-}}
\left\vert u_{n}\right\vert _{p(\cdot )},\left\vert \nabla \frac{1}{p(x)}
 u_{n}\right\vert _{p(\cdot )}\leq \frac{1}{p^{-}}\left\vert \nabla
u_{n}\right\vert _{p(\cdot )}+C\left\vert u_{n}\right\vert _{p(\cdot )}.
\end{equation*}

Thus, $\left\Vert \frac{1}{p(x)}u_{n}\right\Vert \leq C\left\Vert
u_{n}\right\Vert $. Therefore, $(\varphi ^{\prime }(u_{n}),\frac{1}{p(x)}
u_{n})\rightarrow 0$. We may now assume that
\begin{eqnarray*}
c+1 &\geq &\varphi (u_{n})-(\varphi ^{\prime }(u_{n}),\frac{1}{p(x)}u_{n}) \\
&=&\int_{\mathbb{R}^{N}}\frac{1}{p(x)}(\left\vert \nabla u_{n}\right\vert
^{p(x)}+V(x)\left\vert u_{n}\right\vert ^{p(x)})dx-\int_{\mathbb{R}
^{N}}F(x,u_{n})dx \\
&&-\{\int_{\mathbb{R}^{N}}\frac{1}{p(x)}(\left\vert \nabla u_{n}\right\vert
^{p(x)}+V(x)\left\vert u_{n}\right\vert ^{p(x)})dx-\int_{\mathbb{R}^{N}}
\frac{1}{p(x)}f(x,u_{n})u_{n}dx \\
&&-\int_{\mathbb{R}^{N}}\frac{1}{p^{2}(x)}u_{n}\left\vert \nabla
u_{n}\right\vert ^{p(x)-2}\nabla u_{n}\nabla pdx\} \\
&\geq &\int_{\mathbb{R}^{N}}\frac{1}{p^{2}(x)}u_{n}\left\vert \nabla
u_{n}\right\vert ^{p(x)-2}\nabla u_{n}\nabla pdx+\int_{\mathbb{R}^{N}}\{
\frac{1}{p(x)}f(x,u_{n})u_{n}-F(x,u_{n})\}dx.
\end{eqnarray*}
Hence,
\begin{eqnarray} \label{a1}
\int_{\mathbb{R}^{N}}\{\frac{f(x,u_{n})u_{n}}{p(x)}-F(x,u_{n})\}dx &\leq
&C_{1}(\int_{\mathbb{R}^{N}}\left\vert u_{n}\right\vert \left\vert \nabla
u_{n}\right\vert ^{p(x)-1}dx+1)  \notag \\
&\leq &\sigma \int_{\mathbb{R}^{N}}\frac{\left\vert \nabla u_{n}\right\vert
^{p(x)}}{\ln (e+\left\vert u_{n}\right\vert )}dx+C_{1}  \notag \\
&&+C(\sigma )\int_{\mathbb{R}^{N}}\left\vert u_{n}\right\vert ^{p(x)}[\ln
(e+\left\vert u_{n}\right\vert )]^{p(x)-1}dx,
\end{eqnarray}
where $\sigma $ is a small enough positive constant.

Now, we claim that
\begin{equation} \label{abc5}
\int_{\mathbb{R}^{N}}\frac{\left\vert f(x,u_{n})u_{n}\right\vert }{\ln
(e+\left\vert u_{n}\right\vert )}dx\leq c_{1}\int_{\mathbb{R}^{N}}\left\vert
u_{n}\right\vert ^{p(x)}[\ln (e+\left\vert u_{n}\right\vert
)]^{p(x)-1}dx+c_{2}. 
\end{equation}
Note that $\frac{u_{n}}{\ln (e+\left\vert u_{n}\right\vert )}\in X$, and $
\left\Vert \frac{u_{n}}{\ln (e+\left\vert u_{n}\right\vert )}\right\Vert
\leq C_{2}\left\Vert u_{n}\right\Vert $. Let $\frac{u_{n}}{\ln (e+\left\vert
u_{n}\right\vert )}$ be a test function, we have
\begin{eqnarray*}
&&\int_{\mathbb{R}^{N}}f(x,u_{n})\frac{u_{n}}{\ln (e+\left\vert
u_{n}\right\vert )}dx \\
&=&\int_{\mathbb{R}^{N}}\left\vert \nabla u_{n}\right\vert ^{p(x)-2}\nabla
u_{n}\nabla \frac{u_{n}}{\ln (e+\left\vert u_{n}\right\vert )}dx+\int_{
\mathbb{R}^{N}}\frac{V(x)\left\vert u_{n}\right\vert ^{p(x)}}{\ln
(e+\left\vert u_{n}\right\vert )}dx+o(1) \\
&=&\int_{\mathbb{R}^{N}}\frac{\left\vert \nabla u_{n}\right\vert
^{p(x)}+V(x)\left\vert u_{n}\right\vert ^{p(x)}}{\ln (e+\left\vert
u_{n}\right\vert )}dx-\int_{\mathbb{R}^{N}}\frac{\left\vert u_{n}\right\vert
\left\vert \nabla u_{n}\right\vert ^{p(x)}}{(e+\left\vert u_{n}\right\vert
)[\ln (e+\left\vert u_{n}\right\vert )]^{2}}dx+o(1).
\end{eqnarray*}
It is easy to check that $\frac{\left\vert u_{n}\right\vert \left\vert
\nabla u_{n}\right\vert ^{p(x)}}{(e+\left\vert u_{n}\right\vert )[\ln
(e+\left\vert u_{n}\right\vert )]^{2}}\leq \frac{1}{2}\frac{\left\vert
\nabla u_{n}\right\vert ^{p(x)}}{\ln (e+\left\vert u_{n}\right\vert )}$.
Thus, we have
\begin{eqnarray} \label{a2}
C_{3}\int_{\mathbb{R}^{N}}\frac{\left\vert \nabla u_{n}\right\vert
^{p(x)}+V(x)\left\vert u_{n}\right\vert ^{p(x)}}{\ln (e+\left\vert
u_{n}\right\vert )}dx-C_{4} &\leq &\int_{\mathbb{R}^{N}}f(x,u_{n})\frac{u_{n}
}{\ln (e+\left\vert u_{n}\right\vert )}dx  \notag \\
&\leq &C_{5}\int_{\mathbb{R}^{N}}\frac{\left\vert \nabla u_{n}\right\vert
^{p(x)}+V(x)\left\vert u_{n}\right\vert ^{p(x)}}{\ln (e+\left\vert
u_{n}\right\vert )}dx+C_{6}. 
\end{eqnarray}
By $(H_{0})$, for any given positive constant $M$, there exist a positive constant
$c_{M}$ such that
\begin{equation} \label{abc6}
\left\vert f(x,u)u\right\vert +\left\vert F(x,u)\right\vert \leq
c_{M}\left\vert u\right\vert ^{p(x)},\forall x\in
\mathbb{R}
^{N},\forall \left\vert u\right\vert \leq M. 
\end{equation}
From (\ref{a1}), (\ref{a2}), (\ref{abc6}) and condition $(H_{1})$, we have
\begin{eqnarray*}
&&\int_{\left\vert u_{n}\right\vert \geq M}f(x,u_{n})\frac{u_{n}}{\ln
(e+\left\vert u_{n}\right\vert )}dx-\int_{\left\vert u_{n}\right\vert
<M}c_{M}\left\vert u_n\right\vert ^{p(x)}dx \\
&&\overset{(H_{1}),\text{(\ref{abc6})}}{\leq }C_{7}\int_{\mathbb{R}^{N}}\{
\frac{f(x,u_{n})u_{n}}{p(x)}-F(x,u_{n})\}dx \\
&&\overset{(\ref{a1})}{\leq }C_{7}\{\sigma \int_{\mathbb{R}^{N}}\frac{
\left\vert \nabla u_{n}\right\vert ^{p(x)}}{\ln (e+\left\vert
u_{n}\right\vert )}dx+C_{8}+C(\sigma )\int_{\mathbb{R}^{N}}\left\vert
u_{n}\right\vert ^{p(x)}[\ln (e+\left\vert u_{n}\right\vert )]^{p(x)-1}dx\}
\\
&&\overset{(\ref{a2})}{\leq }\frac{1}{2}\int_{\mathbb{R}^{N}}\frac{
f(x,u_{n})u_{n}-C_3 V(x)\left\vert u_{n}\right\vert ^{p(x)}}{\ln (e+\left\vert
u_{n}\right\vert )}dx+C_{10}\int_{\mathbb{R}^{N}}\left\vert u_{n}\right\vert
^{p(x)}[\ln (e+\left\vert u_{n}\right\vert )]^{p(x)-1}dx+C_{9} \\
&&\overset{\text{(\ref{abc6})}}{\leq }\frac{1}{2}\int_{\left\vert
u_{n}\right\vert \geq M}f(x,u_{n})\frac{u_{n}}{\ln (e+\left\vert
u_{n}\right\vert )}dx+\frac{1}{2}\int_{\left\vert u_{n}\right\vert
<M}c_{M}\left\vert u_n\right\vert ^{p(x)}dx \\
&&-\frac{C_3}{2}\int_{\mathbb{R}^{N}}\frac{V(x)\left\vert u_{n}\right\vert ^{p(x)}}{\ln
(e+\left\vert u_{n}\right\vert )}dx+C_{10}\int_{\mathbb{R}^{N}}\left\vert
u_{n}\right\vert ^{p(x)}[\ln (e+\left\vert u_{n}\right\vert
)]^{p(x)-1}dx+C_{9}.
\end{eqnarray*}
Thus, there holds
\begin{eqnarray}  \label{abc3}
&&\int_{\left\vert u_{n}\right\vert \geq M}f(x,u_{n})\frac{u_{n}}{\ln
(e+\left\vert u_{n}\right\vert )}dx+C_3\int_{\mathbb{R}^{N}}\frac{
V(x)\left\vert u_{n}\right\vert ^{p(x)}}{\ln (e+\left\vert u_{n}\right\vert )
}dx-3c_{M}\int_{\left\vert u_{n}\right\vert <M}\left\vert u_n\right\vert
^{p(x)}dx  \notag \\
&\leq &2C_{10}\int_{\mathbb{R}^{N}}\left\vert \nabla p\right\vert \left\vert
u_{n}\right\vert ^{p(x)}[\ln (e+\left\vert u_{n}\right\vert
)]^{p(x)-1}dx+2C_{9}.
\end{eqnarray}
From condition $(V)$, when $R_{0}>0$ is large enough, we have $V(x)\geq
3(M+1)c_{M}/C_3$ for any $\left\vert x\right\vert \geq R_{0}$. Thus, we have
\begin{eqnarray} \label{abc4}
&&C_3\int_{\mathbb{R}^{N}}\frac{V(x)\left\vert u_{n}\right\vert ^{p(x)}}{\ln
(e+\left\vert u_{n}\right\vert )}dx-3c_{M}\int_{\left\vert u_{n}\right\vert
<M}\left\vert u_n\right\vert ^{p(x)}dx  \notag \\
&=&C_3\int_{\left\vert x\right\vert \leq R_{0}}\frac{V(x)\left\vert
u_{n}\right\vert ^{p(x)}}{\ln (e+\left\vert u_{n}\right\vert )}
dx+C_3\int_{\left\vert x\right\vert >R_{0}}\frac{V(x)\left\vert
u_{n}\right\vert ^{p(x)}}{\ln (e+\left\vert u_{n}\right\vert )}dx  \notag \\
&&-3c_{M}\int_{\left\vert u_{n}\right\vert <M,\left\vert x\right\vert \leq
R_{0}}\left\vert u_n\right\vert ^{p(x)}dx-3c_{M}\int_{\left\vert
u_{n}\right\vert <M,\left\vert x\right\vert >R_{0}}\left\vert u_n\right\vert
^{p(x)}dx  \notag \\
&\geq &\int_{\left\vert u_{n}\right\vert <M,\left\vert x\right\vert >R_{0}}
\frac{c_{M}\left\vert u_{n}\right\vert ^{p(x)}}{\ln (e+\left\vert
u_{n}\right\vert )}dx-c_{2}  \notag \\
&&\overset{(\ref{abc6})}{\geq }\int_{\left\vert u_{n}\right\vert <M}\frac{
\left\vert f(x,u_{n})u_{n}\right\vert }{\ln (e+\left\vert u_{n}\right\vert )}
dx-c_{3}. 
\end{eqnarray}
Combining \eqref{abc3} and \eqref{abc4}, we have
\begin{equation*}
\int_{\mathbb{R}^{N}}\frac{\left\vert f(x,u_{n})u_{n}\right\vert }{\ln
(e+\left\vert u_{n}\right\vert )}dx\leq 2C_{10}\int_{\mathbb{R}
^{N}}\left\vert u_{n}\right\vert ^{p(x)}[\ln (e+\left\vert u_{n}\right\vert
)]^{p(x)-1}dx+c_{4}.
\end{equation*}
Therefore, \eqref{abc5} is valid.

\textit{Claim 1.} $\int_{\mathbb{R}^{N}}\frac{\left\vert
f(x,u_{n})u_{n}\right\vert }{\ln (e+\left\vert u_{n}\right\vert )}
dx\rightarrow +\infty $.

We suppose the contrary. Up to a sequence, we can see $\left\{ \int_{\mathbb{
R}^{N}}\frac{\left\vert f(x,u_{n})u_{n}\right\vert }{\ln (e+\left\vert
u_{n}\right\vert )}dx\right\} $ is bounded.

Let $\varepsilon >0$ satisfy $\varepsilon <\min \{1,p^{-}-1,\frac{1}{p^{\ast
+}},(p^{\ast }-\alpha )^{-}\}$. Since $\left\Vert \varphi ^{\prime
}(u_{n})\right\Vert \left\Vert u_{n}\right\Vert \rightarrow 0$ and $\|u_n\|\rightarrow +\infty$, we have
\begin{eqnarray*}
&&\int_{\mathbb{R}^{N}}(\left\vert \nabla u_{n}\right\vert
^{p(x)}+V(x)\left\vert u_{n}\right\vert ^{p(x)})dx \\
&=&\int_{\mathbb{R}^{N}}f(x,u_{n})u_{n}dx+o(1) \\
&\leq &\int_{\mathbb{R}^{N}}\left\vert f(x,u_{n})u_{n}\right\vert
^{\varepsilon }[\ln (e+\left\vert u_{n}\right\vert )]^{1-\varepsilon }\left[
\frac{\left\vert f(x,u_{n})u_{n}\right\vert }{\ln (e+\left\vert
u_{n}\right\vert )}\right] ^{1-\varepsilon }dx+o(1) \\
&\leq &c_{5}(1+\left\Vert u_{n}\right\Vert )^{1+\varepsilon }\int_{\mathbb{R}
^{N}}\frac{\left\vert f(x,u_{n})u_{n}\right\vert ^{\varepsilon
}(1+\left\vert u_{n}\right\vert ^{\varepsilon ^{2}})}{(1+\left\Vert
u_{n}\right\Vert )^{1+\varepsilon }}\left[ \frac{\left\vert
f(x,u_{n})u_{n}\right\vert }{\ln (e+\left\vert u_{n}\right\vert )}\right]
^{1-\varepsilon }dx+o(1) \\
&=&c_{5}(1+\left\Vert u_{n}\right\Vert )^{1+\varepsilon }\int_{\mathbb{R}
^{N}}\frac{\left\vert f(x,u_{n})u_{n}\right\vert ^{\varepsilon }+[\left\vert
f(x,u_{n})\right\vert \left\vert u_{n}\right\vert ^{1+\varepsilon
}]^{\varepsilon }}{(1+\left\Vert u_{n}\right\Vert )^{1+\varepsilon }}\left[
\frac{\left\vert f(x,u_{n})u_{n}\right\vert }{\ln (e+\left\vert
u_{n}\right\vert )}\right] ^{1-\varepsilon }dx+o(1) \\
&\leq &c_{5}(1+\left\Vert u_{n}\right\Vert )^{1+\varepsilon }\int_{\mathbb{R}
^{N}}\left[ \left( \frac{\left\vert f(x,u_{n})u_{n}\right\vert }{
(1+\left\Vert u_{n}\right\Vert )^{\frac{1+\varepsilon }{\varepsilon }}}
\right) ^{\varepsilon }+\left( \frac{\left\vert f(x,u_{n})\right\vert
\left\vert u_{n}\right\vert ^{1+\varepsilon }}{(1+\left\Vert
u_{n}\right\Vert )^{\frac{1+\varepsilon }{\varepsilon }}}\right)
^{\varepsilon }\right] \left[ \frac{\left\vert f(x,u_{n})u_{n}\right\vert }{
\ln (e+\left\vert u_{n}\right\vert )}\right] ^{1-\varepsilon }dx+o(1) \\
&\leq &c_{6}(1+\left\Vert u_{n}\right\Vert )^{1+\varepsilon }+c_{7}.
\end{eqnarray*}

It is a contradiction. Therefore, \textit{Claim 1} is valid.

From \textit{Claim 1} and (\ref{abc5}), we can see
\begin{equation} \label{abc8}
\int_{\mathbb{R}^{N}}\left\vert u_{n}\right\vert ^{p(x)}[\ln (e+\left\vert
u_{n}\right\vert )]^{p(x)-1}dx\rightarrow +\infty \text{ as }n\rightarrow
+\infty. 
\end{equation}

From hypothesis $(H_{1})$ and (\ref{abc5}), we can see
\begin{eqnarray} \label{aj1}
\int_{\left\vert u_{n}\right\vert \geq M}\left\vert u_{n}\right\vert
^{p(x)}[\ln (e+\left\vert u_{n}\right\vert )]^{a(x)-1}dx &\leq &C_{11}\int_{
\mathbb{R}^{N}}\frac{\left\vert f(x,u_{n})u_{n}\right\vert }{\ln
(e+\left\vert u_{n}\right\vert )}dx  \notag \\
&\leq &C_{12}\int_{\mathbb{R}^{N}}\left\vert u_{n}\right\vert ^{p(x)}[\ln
(e+\left\vert u_{n}\right\vert )]^{p(x)-1}dx+C_{12}. 
\end{eqnarray}

Note that $a>>p$ in $\mathbb{R}
^{N}$, then there is a positive constant $M^{\#}>M$ (where $M$ is defined in
(H$_{1}$)) such that$\ $
\begin{equation*}
\lbrack \ln (e+\left\vert t\right\vert )]^{a(x)-p(x)}\geq 4C_{12},\text{ }
\forall \text{ }\left\vert t\right\vert \geq M^{\#},\forall x\in
\mathbb{R}^{N}.
\end{equation*}

\textit{Claim 2.} $\underset{n\rightarrow \infty }{\overline{\lim }}
\int_{\left\vert u_{n}\right\vert \geq M^{\#}}\left\vert u_{n}\right\vert
^{p(x)}[\ln (e+\left\vert u_{n}\right\vert )]^{p(x)-1}dx<\frac{1}{2}\int_{
\mathbb{R}^{N}}\left\vert u_{n}\right\vert ^{p(x)}[\ln (e+\left\vert
u_{n}\right\vert )]^{p(x)-1}dx$.

We argue by contradiction. Up to a sequence, we can see
\begin{equation}\label{aj2}
\int_{\left\vert u_{n}\right\vert \geq M^{\#}}\left\vert u_{n}\right\vert
^{p(x)}[\ln (e+\left\vert u_{n}\right\vert )]^{p(x)-1}dx\geq \frac{1}{3}
\int_{\mathbb{R}^{N}}\left\vert u_{n}\right\vert ^{p(x)}[\ln (e+\left\vert
u_{n}\right\vert )]^{p(x)-1}dx.  
\end{equation}
Combining (\ref{abc8}) and $(3.9)$, we can see
\begin{equation} \label{abc9}
\int_{\left\vert u_{n}\right\vert \geq M^{\#}}\left\vert u_{n}\right\vert
^{p(x)}[\ln (e+\left\vert u_{n}\right\vert )]^{p(x)-1}dx\rightarrow +\infty
\text{ as }n\rightarrow +\infty . 
\end{equation}
According to the defifnition of $M^{\#}$ and (\ref{aj1}), we can see
\begin{eqnarray*}
&&4C_{12}\int_{\left\vert u_{n}\right\vert \geq M^{\#}}\left\vert
u_{n}\right\vert ^{p(x)}[\ln (e+\left\vert u_{n}\right\vert )]^{p(x)-1}dx \\
&\leq &\int_{\left\vert u_{n}\right\vert \geq M^{\#}}\left\vert
u_{n}\right\vert ^{p(x)}[\ln (e+\left\vert u_{n}\right\vert )]^{a(x)-1}dx \\
&\leq &\int_{\left\vert u_{n}\right\vert \geq M}\left\vert u_{n}\right\vert
^{p(x)}[\ln (e+\left\vert u_{n}\right\vert )]^{a(x)-1}dx \\
&\leq &C_{12}\int_{\mathbb{R}^{N}}\left\vert u_{n}\right\vert ^{p(x)}[\ln
(e+\left\vert u_{n}\right\vert )]^{p(x)-1}dx+C_{12} \\
&&\overset{(3.9)}{\leq }3C_{12}\int_{\left\vert u_{n}\right\vert \geq
M^{\#}}\left\vert u_{n}\right\vert ^{p(x)}[\ln (e+\left\vert
u_{n}\right\vert )]^{p(x)-1}dx+C_{12}.
\end{eqnarray*}
Therefore,
$\int_{\left\vert u_{n}\right\vert \geq M^{\#}}\left\vert u_{n}\right\vert
^{p(x)}[\ln (e+\left\vert u_{n}\right\vert )]^{p(x)-1}dx\leq C_{12}$,
which contradicts (\ref{abc9}). Therefore, \textit{Claim 2} is valid.

From \textit{Claim 2}, for large enough $n$, we have
\begin{equation} \label{abc10}
\int_{\left\vert u_{n}\right\vert <M^{\#}}\left\vert u_{n}\right\vert
^{p(x)}[\ln (e+\left\vert u_{n}\right\vert )]^{p(x)-1}dx\geq \frac{1}{2}
\int_{\mathbb{R}^{N}}\left\vert u_{n}\right\vert ^{p(x)}[\ln (e+\left\vert
u_{n}\right\vert )]^{p(x)-1}dx. 
\end{equation}

Notice that $V(x)\rightarrow +\infty $ as $\left\vert x\right\vert
\rightarrow +\infty $.

From (\ref{a2}), (\ref{abc5}) and (\ref{abc10}), we have
\begin{eqnarray*}
&&\int_{\mathbb{R}^{N}}\frac{\left\vert \nabla u_{n}\right\vert
^{p(x)}+V(x)\left\vert u_{n}\right\vert ^{p(x)}}{\ln (e+\left\vert
u_{n}\right\vert )}dx \\
&&\overset{(\ref{a2})}{\leq }c_{8}\int_{\mathbb{R}^{N}}f(x,u_{n})\frac{u_{n}
}{\ln (e+\left\vert u_{n}\right\vert )}dx+C_{4} \\
&\leq &c_{8}\int_{\mathbb{R}^{N}}\frac{\left\vert f(x,u_{n})u_{n}\right\vert
}{\ln (e+\left\vert u_{n}\right\vert )}dx+C_{4} \\
&&\overset{(\ref{abc5})}{\leq }c_{9}\int_{\mathbb{R}^{N}}\left\vert
u_{n}\right\vert ^{p(x)}[\ln (e+\left\vert u_{n}\right\vert
)]^{p(x)-1}dx+C_{12} \\
&&\overset{(\ref{abc10})}{\leq }2c_{9}\int_{\left\vert u_{n}\right\vert
<M^{\#}}\left\vert u_{n}\right\vert ^{p(x)}[\ln (e+\left\vert
u_{n}\right\vert )]^{p(x)-1}dx+C_{12} \\
&\leq &2c_{9}\int_{\left\vert u_{n}\right\vert <M^{\#},\left\vert
x\right\vert \geq R_{0}}\frac{\left\vert u_{n}\right\vert ^{p(x)}}{\ln
(e+\left\vert u_{n}\right\vert )}[\ln (e+\left\vert u_{n}\right\vert
)]^{p(x)}dx \\
&&+2c_{9}\int_{\left\vert u_{n}\right\vert <M^{\#},\left\vert x\right\vert
<R_{0}}\frac{\left\vert u_{n}\right\vert ^{p(x)}}{\ln (e+\left\vert
u_{n}\right\vert )}[\ln (e+\left\vert u_{n}\right\vert )]^{p(x)}dx+C_{13} \\
&\leq &\frac{1}{2}\int_{\left\vert u_{n}\right\vert \leq M^{\#},\left\vert
x\right\vert \geq R_{0}}\frac{V(x)\left\vert u_{n}\right\vert ^{p(x)}}{\ln
(e+\left\vert u_{n}\right\vert )}dx+C_{14} \\
&\leq &\frac{1}{2}\int_{\mathbb{R}^{N}}\frac{\left\vert \nabla
u_{n}\right\vert ^{p(x)}+V(x)\left\vert u_{n}\right\vert ^{p(x)}}{\ln
(e+\left\vert u_{n}\right\vert )}dx+C_{14},
\end{eqnarray*}
where $R_{0}>0$ be large enough such that
\begin{equation*}
\underset{\left\vert x\right\vert \geq R_{0}}{\inf }V(x)\geq
4M^{\#p^{+}}c_{9}.
\end{equation*}

Therefore, $\int_{\mathbb{R}^{N}}\frac{\left\vert \nabla u_{n}\right\vert
^{p(x)}+V(x)\left\vert u_{n}\right\vert ^{p(x)}}{\ln (e+\left\vert
u_{n}\right\vert )}dx$ is bounded, and then $\int_{\mathbb{R}^{N}}\frac{
\left\vert f(x,u_{n})u_{n}\right\vert }{\ln (e+\left\vert u_{n}\right\vert )}
dx$ is bounded. It is a contradiction.

Summarizing the above discussion, we obtain that \{$u_{n}$\} is bounded. Therefore, the proof of Lemma 3.3 is completed.
\end{proof}

Next we give the proof of Theorem 1.1.

\textbf{Proof of Theorem 1.1}. We only prove the existence and asymptotic
behavior of positive solution $u_{1}$. The rest is similar. We divide the proof
into three steps as follows.

\textit{Step 1}. We show the existence of a nontrivial nonnegative weak
solution $u_{1}$.

Denote
\begin{equation*}
f^{+}(x,u)=\left\{
\begin{array}{c}
f(x,u),\, u\geq 0 \\
0,\qquad\,\,\, u<0
\end{array}
\right. \text{, }F^{+}(x,u)=\int_{0}^{u}f^{+}(x,t)dt.
\end{equation*}

Consider the auxiliary problem
\begin{equation*}
\text{(P}^{+}\text{) }\left\{
\begin{array}{l}
-\bigtriangleup _{p(x)}u+V(x)\left\vert u\right\vert ^{p(x)-2}u=f^{+}(x,u),
\text{ in }
\mathbb{R}^{N}, \\
u\in X.
\end{array}
\right.
\end{equation*}

The corresponding functional of (P$^{+}$) is
\begin{equation*}
\varphi ^{+}\left( u\right) =\int_{\mathbb{R}^{N}}\frac{1}{p(x)}(\left\vert
\nabla u\right\vert ^{p(x)}+V(x)\left\vert u\right\vert ^{p(x)})dx-\int_{
\mathbb{R}^{N}}F^{+}(x,u)dx,\forall u\in X.
\end{equation*}

Obviously, $\varphi ^{+}$ is $C^{1}$ in $X$.

Let's show that $\varphi ^{+}$ satisfies conditions of the well-known Mountain Pass
Lemma. Similar to the proof of Lemma 3.3, we can see that $\varphi ^{+}$
satisfy Cerami condition. Since $p(x)<\alpha (x)<<p^{\ast }(x),$ the
embedding $X\hookrightarrow L^{\alpha (\cdot )}(\Omega )$ is compact, then
there exists $C_{0}>0$ such that
\begin{equation*}
\left\vert u\right\vert _{p(\cdot )}\leq C_{0}\left\Vert u\right\Vert \text{
, }\forall u\in X.
\end{equation*}

By the assumptions $(H_{0})$ and $(H_{2})$, we have
\begin{equation*}
F^{+}(x,t)\leq \sigma \frac{1}{p(x)}\left\vert t\right\vert ^{p(x)}+C(\sigma
)\left\vert t\right\vert ^{\alpha (x)}\text{, }\forall (x,t)\in \mathbb{R}
^{N}\times
\mathbb{R}.
\end{equation*}

Let $\sigma \in (0,\frac{1}{4}V_{0})$, where $V_{0}$ is defined in (V). We
have
\begin{eqnarray*}
&&\int_{\mathbb{R}^{N}}\frac{1}{p(x)}(\left\vert \nabla u\right\vert
^{p(x)}+V(x)\left\vert u\right\vert ^{p(x)})dx-\sigma \int_{\mathbb{R}^{N}}
\frac{1}{p(x)}\left\vert u\right\vert ^{p(x)}dx \\
&\geq &\frac{3}{4}\int_{\mathbb{R}^{N}}\frac{1}{p(x)}(\left\vert \nabla
u\right\vert ^{p(x)}+V(x)\left\vert u\right\vert ^{p(x)})dx.
\end{eqnarray*}

Since $\alpha \in C(\mathbb{R}^{N})$ and $p(x)<<\alpha (x)<<p^{\ast }(x)$,
we can divide the domain $\mathbb{R}^{N}$ into a sequence of disjoint small
cubes $\Omega _{i}$ ($i=1,\cdots ,\infty $) each one has the same
side length such that $\mathbb{R}^{N}=\underset{i=1}{\overset{\infty }{\cup
}}\overline{\Omega _{i}}$ and
\begin{equation*}
\underset{\Omega _{i}}{\sup }\, p(x)<\underset{\Omega _{i}}{\inf }\, \alpha
(x)\leq \underset{\Omega _{i}}{\sup }\, \alpha (x)<\underset{\Omega _{i}}{\inf }\,
p^{\ast }(x).
\end{equation*}

We introduce a number $\epsilon$ as follows:
\begin{equation*}
\epsilon :=\underset{1\leq i\leq \infty }{\inf }\{\underset{\Omega _{i}}{\inf
}\, \alpha (x)-\underset{\Omega _{i}}{\sup }\, p(x)\}.
\end{equation*}
According to assumptions $(p)$ and $(H_0)$, one easily see that $\epsilon>0$ as long as
the side length of $\Omega_i$ is made sufficiently small.

Denote $\left\Vert u\right\Vert _{\Omega _{i}}$ the Sobolev norm of $u$ on $
\Omega _{i}$, i.e.,
\begin{equation*}
\left\Vert u\right\Vert _{\Omega _{i}}=\inf \left\{ \lambda >0\left\vert
\int_{\Omega _{i}}\frac{1}{p(x)}\left( \left\vert \frac{\nabla u}{\lambda }
\right\vert ^{p(x)}+V(x)\left\vert \frac{u}{\lambda }\right\vert
^{p(x)}\right) dx\leq 1\right. \right\}
\end{equation*}
and $\left\vert u\right\vert _{\alpha (\cdot ),\Omega _{i}}$ the Lebesgue
norm of $u$ on $\Omega _{i}$, i.e.,
\begin{equation*}
\left\vert u\right\vert _{\alpha (\cdot ),\Omega _{i}}=\inf \left\{ \lambda
>0\left\vert \int_{\Omega _{i}}\left\vert \frac{u}{\lambda }\right\vert
^{\alpha (x)}dx\leq 1\right. \right\} .
\end{equation*}

It is easy to see that $\left\Vert u\right\Vert _{\Omega _{i}}\leq
\left\Vert u\right\Vert $, and there exist $\xi _{i},\eta _{i}\in \overline{
\Omega _{i}}$ such that
\begin{eqnarray*}
\left\vert u\right\vert _{\alpha (\cdot ),\Omega _{i}}^{\alpha (\xi _{i})}
&=&\int_{\Omega _{i}}\left\vert u\right\vert ^{\alpha (x)}dx, \\
\left\Vert u\right\Vert _{\Omega _{i}}^{p(\eta _{i})} &=&\int_{\Omega _{i}}(
\frac{1}{p(x)}\left\vert \nabla u\right\vert ^{p(x)}+\frac{V(x)}{p(x)}
\left\vert u\right\vert ^{p(x)})dx.
\end{eqnarray*}

When $\left\Vert u\right\Vert $ is small enough, we have
\begin{eqnarray*}
C(\sigma )\int_{\Omega }\left\vert u\right\vert ^{\alpha (x)}dx &=&C(\sigma )
\underset{i=1}{\overset{\infty }{\sum }}\int_{\Omega _{i}}\left\vert
u\right\vert ^{\alpha (x)}dx \\
&=&C(\sigma )\underset{i=1}{\overset{\infty }{\sum }}\left\vert u\right\vert
_{\alpha (\cdot ),\Omega _{i}}^{\alpha (\xi _{i})}\text{ (where }\xi _{i}\in
\overline{\Omega _{i}}\text{)} \\
&\leq &C\underset{i=1}{\overset{\infty }{\sum }}\left\Vert u\right\Vert
_{\Omega _{i}}^{\alpha (\xi _{i})}\text{ (by Corollary 8.3.2 of \cite{j5})}
\\
&\leq &C\left\Vert u\right\Vert ^{\epsilon }\underset{i=1}{\overset{\infty }{
\sum }}\left\Vert u\right\Vert _{\Omega _{i}}^{p(\eta _{i})}\text{ (where }
\eta _{i}\in \overline{\Omega _{i}}\text{)} \\
&=&C\left\Vert u\right\Vert ^{\epsilon }\underset{i=1}{\overset{\infty }{
\sum }}\int_{\Omega _{i}}(\frac{1}{p(x)}\left\vert \nabla u\right\vert
^{p(x)}+\frac{V(x)}{p(x)}\left\vert u\right\vert ^{p(x)})dx \\
&=&C\left\Vert u\right\Vert ^{\epsilon }\int_{\mathbb{R}^{N}}(\frac{1}{p(x)}
\left\vert \nabla u\right\vert ^{p(x)}+\frac{V(x)}{p(x)}\left\vert
u\right\vert ^{p(x)})dx \\
&\leq &\frac{1}{4}\int_{\mathbb{R}^{N}}(\frac{1}{p(x)}\left\vert \nabla
u\right\vert ^{p(x)}+\frac{V(x)}{p(x)}\left\vert u\right\vert ^{p(x)})dx.
\end{eqnarray*}

Thus, there holds
\begin{eqnarray*}
\varphi ^{+}(u) &\geq &\int_{\mathbb{R}^{N}}(\frac{1}{p(x)}\left\vert \nabla
u\right\vert ^{p(x)}+\frac{V(x)}{p(x)}\left\vert u\right\vert
^{p(x)})dx-\sigma \int_{\mathbb{R}^{N}}\frac{1}{p(x)}\left\vert u\right\vert
^{p(x)}dx-C(\sigma )\int_{\mathbb{R}^{N} }\left\vert u\right\vert ^{\alpha (x)}dx \\
&\geq &\frac{1}{2}\int_{\mathbb{R}^{N}}(\frac{1}{p(x)}\left\vert \nabla
u\right\vert ^{p(x)}+\frac{V(x)}{p(x)}\left\vert u\right\vert ^{p(x)})dx
\text{ when }\|u\| \text{ is small enough.}
\end{eqnarray*}

Therefore, there exist $r>0$ and $\delta >0$ such that $\varphi (u)\geq
\delta >0$ for every $u\in X$ and $\left\Vert u\right\Vert =r.$

From $(H_{1})$, we have
\begin{equation*}
F^{+}(x,t)\geq C_{1}\left\vert t\right\vert ^{p(x)}[\ln (1+\left\vert
t\right\vert )]^{a(x)}-c_{2},\forall (x,t)\in \mathbb{R}^{N}\times
\mathbb{R}^{+}.
\end{equation*}

We may assume there exists $x_{0}\in \Omega $ such that $\nabla p(x_{0})\neq
0$.

Define $h\in C_{0}(\overline{B(x_{0},\varepsilon )})$ as
\begin{equation*}
h(x)=\left\{
\begin{array}{cc}
0, & \left\vert x-x_{0}\right\vert \geq \varepsilon \\
\varepsilon -\left\vert x-x_{0}\right\vert , & \left\vert x-x_{0}\right\vert
<\varepsilon
\end{array}
\right. .
\end{equation*}

From Lemma 2.8, we may let $\varepsilon >0$ is small enough such that
\begin{equation*}
\varphi ^{+}(th)\rightarrow -\infty \text{ as }t\rightarrow +\infty .
\end{equation*}

Since $\varphi ^{+}\left( 0\right) =0,$ $\varphi ^{+}$ satisfies the
conditions of Mountain Pass lemma. So $\varphi ^{+}$ admits at least one
nontrivial critical point, which implies the problem (P$^{+}$) has a
nontrivial weak solution $u_{1}$. It is easy to see that $u_{1}$ is
nonnegative. Therefore, $u_{1}$ is a nontrivial nonnegative weak solution of
$(P)$.

\textit{Step 2.} We obtain the asymptotic behavior of $u_{1}$.

According to the Theorems 2.2 and 3.2 of \cite{n19}, $u_{1}$ is locally
bounded. From Theorem 1.2 of \cite{e10}, $u_{1}$ is locally $C^{1,\alpha }$
continuous. Similar to the proof of Proposition 2.5\ of \cite{zj6}, we
obtain that $u_{1}$ is $C^{1,\alpha }(\mathbb{R}
^{N})$ and satisfies $\underset{\left\vert x\right\vert \rightarrow \infty }{
\lim }u_{1}(x)=0$ and $\underset{\left\vert x\right\vert \rightarrow \infty }
{\lim }\left\vert \nabla u_{1}(x)\right\vert =0$.

\textit{Step 3.} We show that $u_{1}$ is positive.

Noticing that $f$ satisfies $(H_{0})$ and $u_{1}$ is nonnegative. We can see
that
\begin{eqnarray*}
&&-\bigtriangleup _{p(x)}u_{1}+V(x)\left\vert u_{1}\right\vert
^{p(x)-2}u_{1}+C(\left\vert u_{1}\right\vert ^{p(x)-2}u_{1}+\left\vert
u_{1}\right\vert ^{\alpha (x)-2}u_{1}) \\
&=&C(\left\vert u_{1}\right\vert ^{p(x)-2}u_{1}+\left\vert u_{1}\right\vert
^{\alpha (x)-2}u_{1})+f(x,u_{1})\geq 0.
\end{eqnarray*}

Thus, $u_{1}$ is a nontrivial nonnegative weak supersolution of the following
equation
\begin{equation*}
-\bigtriangleup _{p(x)}u_{1}+V(x)\left\vert u_{1}\right\vert
^{p(x)-2}u_{1}+C(\left\vert u_{1}\right\vert ^{p(x)-2}u_{1}+\left\vert
u_{1}\right\vert ^{\alpha (x)-2}u_{1})=0.
\end{equation*}

According to the Theorem 1.1 of \cite{e31}, we can see that $u_{1}>0$ in $
\mathbb{R}^{N}$. By now, we have finished the proof of Theorem 1.1.

Now, we proceed to prove Theorem 1.2. For this purpose, we shall first make some functional-analytic preparations. Note that $X$
is a reflexive and separable Banach space. Therefore, there are $\left\{ e_{j}\right\} \subset X$ and $\left\{
e_{j}^{\ast }\right\} \subset X^{\ast }$ (see \cite{22}, Section 17,
Theorem 2-3) such that
\begin{equation*}
X=\overline{span}\{e_{j}\text{, }j=1,2,\cdots \}\text{, }\left. {}\right.
X^{\ast }=\overline{span}^{W^{\ast }}\{e_{j}^{\ast }\text{, }j=1,2,\cdots \},
\end{equation*}
and
\begin{equation*}
<e_{j}^{\ast },e_{j}>=\left\{
\begin{array}{c}
1,i=j, \\
0,i\neq j.
\end{array}
\right.
\end{equation*}

For convenience, we write $X_{j}=span\{e_{j}\}$, $Y_{k}=\overset{k}{\underset
{j=1}{\oplus }}X_{j}$, and $Z_{k}=\overline{\overset{\infty }{\underset{j=k}{
\oplus }}X_{j}}$.

\begin{lemma}\label{lemma3.4}
Let $\alpha \in C_{+}\left( \mathbb{R}^{N}\right)$, $
\alpha (x)<<p^{\ast }(x)$ for any $x\in \mathbb{R}^{N}$, denote
\begin{equation*}
\beta _{k}=\sup \left\{ \left\vert u\right\vert _{\alpha (\cdot )}\left\vert
\left\Vert u\right\Vert =1,u\in Z_{k}\right. \right\}.
\end{equation*}
Then, $\underset{k\rightarrow \infty }{\lim }\beta _{k}=0$.
\end{lemma}

\begin{proof} Obviously, $0<\beta _{k+1}\leq \beta _{k},$ so $\beta
_{k}\rightarrow \beta \geq 0.$ Let $u_{k}\in Z_{k}$ satisfy
\begin{equation*}
\left\Vert u_{k}\right\Vert =1,0\leq \beta _{k}-\left\vert u_{k}\right\vert
_{\alpha (\cdot )}<\frac{1}{k}.
\end{equation*}
Then there exists a subsequence of $\{u_{k}\}$ (which we still denote by $
\{u_{k}\}$) such that $u_{k}\rightharpoonup u,$ and
\begin{equation*}
<e_{j}^{\ast },u>=\underset{k\rightarrow \infty }{\lim }\left\langle
e_{j}^{\ast },u_{k}\right\rangle =0\text{, }\forall e_{j}^{\ast },
\end{equation*}
which implies that $u=0$, and so $u_{k}\rightharpoonup 0.$ Since the
embedding from $X$ to $L^{\alpha (\cdot )}\left( \mathbb{R}^{N}\right) $ is
compact, then $u_{k}\rightarrow 0$ in $L^{\alpha (\cdot )}\left( \mathbb{R}
^{N}\right) $. Hence we get $\beta _{k}\rightarrow 0$ as $k\rightarrow
\infty $. The proof of Lemma 3.4 is completed.
\end{proof}

To complete the proof of Theorem 1.2, we recall the following critical point lemma (see e.g., \cite[Theorem 4.7]{20a}. If the Cerami condition is replaced by the well known $(P.S.)$-condition, see \cite[Page 221, Theorem 3.6]{e39} for the corresponding version of critical point theorem.

\begin{lemma}
Suppose that $\varphi \in C^{1}(X, \mathbb{R})$ is even, and
satisfies the Cerami condition. Let $V^{+}$, $V^{-}\subset X$ be closed
subspaces of $X$ with codim$V^{+}+1=$dim $V^{-}$, and suppose there holds

($1^{0}$) $\varphi (0)=0$.

($2^{0}$) $\exists \tau >0,$ $\gamma >0$ such that $\forall u\in V^{+}:$ $
\Vert u\Vert =\gamma \Rightarrow \varphi (u)\geq \tau .$

($3^{0}$) $\exists \rho >0\ $such that $\forall u\in V^{-}:$ $\Vert u\Vert
\geq \rho \Rightarrow \varphi (u)\leq 0.$

Consider the following set:
\begin{equation*}
\Gamma =\{g\in C^{0}(X,X)\mid g\text{ is odd, }g(u)=u\text{ if }u\in V^{-}
\text{ and }\Vert u\Vert \geq \rho \},
\end{equation*}
then

($a$) $\forall \delta >0$, $g\in \Gamma $, $S_{\delta }^{+}\cap g(V^{-})\neq
\varnothing $, here $S_{\delta }^{+}=\{u\in V^{+}\mid \Vert u\Vert =\delta
\};$

($b$) the number $\varpi :=\underset{g\in \Gamma }{\inf }\underset{\text{ }
u\in V^{-}}{\sup }\varphi (g(u))\geq \tau >0$ is a critical value for $
\varphi $.\newline
\end{lemma}

\textbf{Proof of Theorem 1.2} We only need to prove the existence of
infinitely many pairs of weak solutions. The proof of regularity and
asymptotic behavior of solutions are similarly to that of Theorem 1.1.

According to $(H_{0})$, $(H_{1})$ and $(H_{3})$, $\varphi $ is an even
functional and satisfies Cerami condition. Let $V_{k}^{+}=Z_{k}$, it is a
closed linear subspace of $X$ and $V_{k}^{+}\oplus Y_{k-1}=X$.

We may assume that there exists $x_{n}\in \Omega $ such that $\nabla
p(x_{n})\neq 0$.

Define $h_{n}\in C_{0}(\overline{B(x_{n},\varepsilon _{n})})$ as
\begin{equation*}
h_{n}(x)=\left\{
\begin{array}{cc}
0, & \left\vert x-x_{n}\right\vert \geq \varepsilon _{n} \\
\varepsilon _{n}-\left\vert x-x_{n}\right\vert , & \left\vert
x-x_{n}\right\vert <\varepsilon _{n}
\end{array}
\right. .
\end{equation*}

From Lemma 2.8, we may let $\varepsilon _{n}>0$ be small enough such that
\begin{equation*}
\varphi (th_{n})\rightarrow -\infty \text{ as }t\rightarrow +\infty .
\end{equation*}

Without loss of generality, we may assume that
\begin{equation*}
supp\,h_{i}\cap supp\, h_{j}=\varnothing \text{, }\forall i\neq j.
\end{equation*}

Set $V_{k}^{-}=span\{h_{1},\cdots ,h_{k}\}$. We will prove that there are
infinite many pairs of $V_{k}^{+}$ and $V_{k}^{-}$, such that $\varphi $
satisfies the conditions of Lemma 3.5 and the corresponding critical value $
\varpi _{k}:=\underset{g\in \Gamma }{\inf }\underset{\text{ }u\in V_{k}^{-}}{
\sup }\varphi (g(u))\rightarrow +\infty $ when $k\rightarrow +\infty $,
which implies that there are infinitely many pairs of solutions to the problem
$(P)$.

For any $k=1,2,\cdots $, we will prove that there exist $\rho _{k}>\gamma
_{k}>0$ and large enough $k$ such that
\begin{eqnarray*}
(A_{1})\text{ }b_{k} &:&=\inf \left\{ \varphi (u)\mid u\in
V_{k}^{+},\left\Vert u\right\Vert =\gamma _{k}\right\} \rightarrow +\infty
\text{ }(k\rightarrow +\infty ); \\
(A_{2})\text{ }a_{k} &:&=\max \left\{ \varphi (u)\right\vert \text{ }u\in
V_{k}^{-},\left\Vert u\right\Vert =\rho _{k}\}\leq 0.
\end{eqnarray*}

First, we show ($A_{1}$) holds. Let $\sigma \in (0,V_{0})$ be small enough, where $V_{0}$ is defined in $(V)$. By $(H_{0})$
and $(H_{2})$, there is a $C(\sigma )>0$ such that
\begin{equation*}
F(x,u)\leq \sigma \left\vert u\right\vert ^{p(x)}+C(\sigma )\left\vert
u\right\vert ^{\alpha (x)},\forall x\in\mathbb{R}
^{N},\forall u\in
\mathbb{R}.
\end{equation*}
By computation, for any $u\in Z_{k}$ with $\left\Vert u\right\Vert =\gamma
_{k}=(2C(\sigma )\alpha ^{+}\beta _{k}^{\alpha ^{+}})^{1/(p^{-}-\alpha ^{+})}
$, we have
\begin{eqnarray*}
\varphi (u) &=&\int_{\mathbb{R}^{N}}\frac{1}{p(x)}(\left\vert \nabla
u\right\vert ^{p(x)}+V(x)\left\vert u\right\vert ^{p(x)})dx-\int_{\mathbb{R}
^{N}}F(x,u)dx \\
&\geq &\frac{1}{p^{+}}\int_{\mathbb{R}^{N}}(\left\vert \nabla u\right\vert
^{p(x)}+V(x)\left\vert u\right\vert ^{p(x)})dx-C(\sigma )\int_{\mathbb{R}
^{N}}\left\vert u\right\vert ^{\alpha (x)}dx-\sigma \int_{\mathbb{R}
^{N}}\left\vert u\right\vert ^{p(x)}dx \\
&\geq &\frac{1}{2p^{+}}\left\Vert u\right\Vert ^{p^{-}}-C(\sigma )\left\vert
u\right\vert _{\alpha (\cdot )}^{\alpha (\xi )}\text{ (where }\xi \in
\mathbb{R}^{N}\text{)} \\
&\geq &\left\{
\begin{array}{l}
\frac{1}{2p^{+}}\left\Vert u\right\Vert ^{p^{-}}-C(\sigma )\text{, if }
\left\vert u\right\vert _{\alpha (\cdot )}\leq 1, \\
\frac{1}{2p^{+}}\left\Vert u\right\Vert ^{p^{-}}-C(\sigma )\beta
_{k}^{\alpha ^{+}}\left\Vert u\right\Vert ^{\alpha ^{+}}\text{, if }
\left\vert u\right\vert _{\alpha (\cdot )}>1,
\end{array}
\right.  \\
&\geq &\frac{1}{2p^{+}}\left\Vert u\right\Vert ^{p^{-}}-C(\sigma )\beta
_{k}^{\alpha ^{+}}\left\Vert u\right\Vert ^{\alpha ^{+}}-C(\sigma ) \\
&=&\frac{1}{2p^{+}}(2C(\sigma )\alpha ^{+}\beta _{k}^{\alpha
^{+}})^{p^{-}/(p^{-}-\alpha ^{+})}-C(\sigma )\beta _{k}^{\alpha
^{+}}(2C(\sigma )\alpha ^{+}\beta _{k}^{\alpha ^{+}})^{\alpha
^{+}/(p^{-}-\alpha ^{+})}-C(\sigma ) \\
&=&\frac{1}{2}(\frac{1}{p^{+}}-\frac{1}{\alpha ^{+}})(2C(\sigma )\alpha
^{+}\beta _{k}^{\alpha ^{+}})^{p^{-}/(p^{-}-\alpha ^{+})}-C(\sigma
)\rightarrow +\infty \text{ (as }k\rightarrow \infty \text{),}
\end{eqnarray*}
because $p^{+}<\alpha ^{+}$ and $\beta _{k}\rightarrow 0^{+}$ as $
k\rightarrow \infty $.
Therefore, $b_{k}\rightarrow +\infty $, (as $k\rightarrow \infty $).

Now, we show that $(A_{2})$ holds. From Lemma 2.8, it is easy to see that
\begin{equation*}
\varphi (th)\rightarrow -\infty \text{ as }t\rightarrow +\infty ,\forall
h\in V_{k}^{-}=span\{h_{1},\cdots ,h_{k}\}\text{ with}\parallel h\parallel
=1,
\end{equation*}
which implies that $(A_{2})$ holds. 

To sum up, the proof of Theorem 1.2 is completed.

\end{document}